\documentclass[12pt]{amsart}
\usepackage{latexsym,tikz,hyperref,cleveref,bm}
\usepackage[hmargin=1in,vmargin=1in]{geometry}

\newcommand{\al}{\alpha}
\newcommand{\be}{\beta}
\newcommand{\ga}{\gamma}
\newcommand{\de}{\delta}

\newcommand{\io}{\iota}

\newcommand{\la}{\lambda}

\newcommand{\om}{\omega}

\newcommand{\si}{\sigma}
\renewcommand{\th}{\theta}
\newcommand{\ze}{\zeta}

\newcommand{\bfe}{{\bf e}}

\newcommand{\bt}{{\bm \theta}}

\newcommand{\bx}{{\bf x}}

\newcommand{\bbC}{{\mathbb C}}

\newcommand{\bbN}{{\mathbb N}}

\newcommand{\bbZ}{{\mathbb Z}}
\newcommand{\cA}{{\mathcal A}}

\newcommand{\cS}{{\mathcal S}}

\newcommand{\zh}{\hat{0}}
\newcommand{\oh}{\hat{1}}

\newcommand{\ptn}{\vdash}

\newcommand{\case}[4]{\left\{\barr{ll}#1&\mbox{#2}\\#3&\mbox{#4}\earr\right.}

\def\<{\langle}
\def\>{\rangle}

\newcommand{\spn}[1]{\langle{#1}\rangle}

\newcommand{\ra}{\rightarrow}

\newcommand{\fS}{{\mathfrak S}}

\newcommand{\St}{\tilde{S}}

\newcommand{\fb}{\ol{f}}

\newcommand{\Sb}{\ol{S}}

\def\multiset#1#2{\ensuremath{\left(\kern-.3em\left(\genfrac{}{}{0pt}{}{#1}{#2}\right)\kern-.3em\right)}}

\newcommand{\ben}{\begin{enumerate}}
\newcommand{\een}{\end{enumerate}}
\newcommand{\ble}{\begin{lem}}
\newcommand{\ele}{\end{lem}}
\newcommand{\bth}{\begin{thm}}
\newcommand{\eth}{\end{thm}}
\newcommand{\bpr}{\begin{prop}}
\newcommand{\epr}{\end{prop}}
\newcommand{\bco}{\begin{cor}}
\newcommand{\eco}{\end{cor}}
\newcommand{\bcon}{\begin{conj}}
\newcommand{\econ}{\end{conj}}
\newcommand{\bde}{\begin{defn}}
\newcommand{\ede}{\end{defn}}
\newcommand{\bex}{\begin{exa}}
\newcommand{\eex}{\end{exa}}

\newcommand{\barr}{\begin{array}}
\newcommand{\earr}{\end{array}}
\newcommand{\btab}{\begin{tabular}}
\newcommand{\etab}{\end{tabular}}
\newcommand{\beq}{\begin{equation}}
\newcommand{\eeq}{\end{equation}}
\newcommand{\bea}{\begin{eqnarray*}}
\newcommand{\eea}{\end{eqnarray*}}
\newcommand{\bal}{\begin{align*}}
\newcommand{\bce}{\begin{center}}
\newcommand{\ece}{\end{center}}
\newcommand{\bpi}{\begin{picture}}
\newcommand{\epi}{\end{picture}}
\newcommand{\bpp}{\begin{picture}}
\newcommand{\epp}{\end{picture}}
\newcommand{\bfi}{\begin{figure} \begin{center}}
\newcommand{\efi}{\end{center} \end{figure}}
\newcommand{\bprf}{\begin{proof}}
\newcommand{\eprf}{\end{proof}\medskip}

\newcommand{\bsl}{\begin{slide}{}}
\newcommand{\esl}{\end{slide}}
\newcommand{\bfr}{\begin{frame}}
\newcommand{\efr}{\end{frame}}

\newcommand{\dil}{\displaystyle}

\DeclareMathOperator{\inv}{inv}
\DeclareMathOperator{\Inv}{Inv}

\DeclareMathOperator{\rk}{rk}

\DeclareMathOperator{\sgn}{sgn}

\newcommand{\comp}{\models}

\newcommand{\hqed}{\hfill \qed}

\newcommand{\eqqed}[1]{$\rule{1ex}{0ex}\hfill{\dil#1}\hfill\qed$}

\newcommand{\ol}{\overline}

\newcommand{\hs}[1]{\hspace{#1}}
\newcommand{\hso}[1]{\hspace{-1pt}}
\newcommand{\vs}[1]{\vspace{#1}}

\newcommand{\sbe}{\subseteq}

\newcommand{\setm}{\setminus}

\newcommand{\iso}{\cong}

\newtheorem{thm}{Theorem}[section]
\newtheorem{prop}[thm]{Proposition}
\newtheorem{cor}[thm]{Corollary}
\newtheorem{lem}[thm]{Lemma}
\newtheorem{conj}[thm]{Conjecture}

\theoremstyle{definition}
\newtheorem{defn}[thm]{Definition}
\newtheorem{exa}[thm]{Example}
\newtheorem{rem}[thm]{Remark}

\newcommand{\Lb}{\ol{L}}
\newcommand{\Pib}{\ol{\Pi}}
\newcommand{\da}{\hs{-2pt}\downarrow}

\DeclareMathOperator{\minb}{minb}
\DeclareMathOperator{\maxb}{maxb}

\DeclareMathOperator{\Hilb}{Hilb}
\DeclareMathOperator{\rev}{rev}
\DeclareMathOperator{\CR}{CR}
\DeclareMathOperator{\RG}{RG}
\DeclareMathOperator{\SRG}{SRG}
\newcommand{\SA}{\cS\cA}
\newcommand{\etalchar}[1]{$^{#1}$}

\begin{document}
\pagestyle{plain}

\title{Stirling numbers for complex reflection groups
}
\author{Bruce E. Sagan}
\address{ Department of Mathematics\\ Michigan State University\\
 East Lansing, MI 48824, USA}
\email{ sagan@math.msu.edu}

\author{Joshua P. Swanson}
\address{Department of Mathematics\\ Univesity of Southern California\\
\small Los Angeles, CA 90007, USA}
\email{swansonj@usc.edu}

\date{\today}


 
\begin{abstract}
In an earlier paper, we defined and studied $q$-analogues of the Stirling numbers of both types for the Coxeter group of type $B$.  In the present work, we show how this approach can be extended to all irreducible complex reflection groups $G$.  The Stirling numbers of the first and second kind are defined via the Whitney numbers of the first and second kind, respectively, of the intersection lattice of $G$.  For the groups
$G(m,p,n)$, these numbers and polynomials can be given combinatorial interpretations in terms of various statistics.  The ordered version of ths $q$-Stirling numbers of the second kind also show up in conjectured Hilbert series for certain super coinvariant algebras.
\end{abstract}

\maketitle

\tableofcontents


\section{Introduction}

The purpose of this work is to define and study Stirling numbers of both kinds for any irreducible complex reflection group $G$.  In order to motivate our construction, we first examine the classic case in type $A$.

We will use the following notation for two subsets of the integers, $\bbZ$,
\begin{align*}
    \bbN&=\{0,1,2,\ldots\},\\
    [n]&=\{1,2,\ldots,n\}.
\end{align*}
 If $S$ is a finite set then a {\em set partition}, $\si$, of $S$ is an unordered collection of nonempty subsets $S_1,S_2,\ldots, S_k$  with disjoint union $\uplus_i S_i = S$.  The $S_i$ are called {\em blocks} and we write $\si=S_1/S_2/\ldots/S_k\ptn S$ where set braces and commas are often eliminated in the $S_i$.  For example, if $S=[9]$ then the partition $\si$ with blocks $\{1,3,4\}$, $\{2,7\}$, and $\{5,6,8\}$ would be written
$\si=134/27/568$.  If we let 
$$
S([n],k) = \{\si\ptn[n] \mid \text{$\si$ has $k$ blocks}\}
$$
then the {\em Stirling numbers of the second kind} are 
$$
S(n,k) = \# S([n],k)
$$
where the hash tag denotes cardinality and $n\in\bbN$, $k\in\bbZ$.  Note that this definition implies that $S(n,k)=0$ for $k<0$ or $k>n$.

Let $\fS_n$ denote the symmetric group of all permutations of $[n]$.  Then every $\pi\in\fS_n$ can be written as a disjoint product of cycles $\pi=c_1 c_2\ldots c_n$.  
Letting
$$
c([n],k) = \{\pi\in\fS_n \mid \text{$\pi$ has $k$ disjoint cycles}\},
$$
the {\em Stirling numbers of the first kind} are
$$
s(n,k) = 
(-1)^{n-k} \# c([n],k).
$$

In order to explain the sign in the definition of $s(n,k)$ as well as our generalization to complex reflection groups, we need the partition lattice.  If $\si=S_1/\ldots/S_k$ and $\tau=T_1/\ldots/T_k$ are partitions of the same set $S$ then $\si$ is a {\em refinement} of $\tau$, written $\si\le\tau$, if every $S_i$ is contained in some $T_j$.
Refinement is a partial order on the partitions of $S$ which is, in fact, a lattice in the order theoretic sense.  Let $\Pi_n$ denote the lattice of partitions of $[n]$.

Suppose that $P$ is a finite poset (partially ordered set) with a unique minimum element $\zh$.  The {\em M\"obius function} of $P$ is $\mu:P\ra\bbZ$ defined inductively by the equation
$$
\mu(x)=\case{1}{if $x=\zh$,}{-\dil\sum_{y<x} \mu(y)}{if $x>\zh$.}
$$
Suppose that $P$ is {\em ranked}, which means that for any $x\in P$, all maximal chains from $\zh$ to $x$ have the same length which is denoted 
$\rk x$ and called the {\em rank of $x$}.  Given $k\in\bbZ$, define the {\em Whitney numbers of the second and first kinds} for $P$ as
$$
W(P,k) = \#\{ x\in P  \mid \rk x = k\}
$$
and
$$
w(P,k) = \sum_{\rk x = k} \mu(x),
$$
respectively.  The following result is well known, for example see~\cite[Section 5.4]{Sag:aoc}.
\begin{thm}\label{thm:Ww-Ss}
\label{Pi_n}
We have

\vs{10pt}

\eqqed{
\text{$W(\Pi_n,k) = S(n,n-k)$ and $w(\Pi_n,k)=s(n,n-k)$.}
}
\end{thm}

To connect the previous theorem with complex reflection groups we will need intersection lattices.  A set $\cA=\{H_1,H_2,\ldots,H_k\}$ of hyperplanes in $\bbC^n$ is called an {\em arrangement}.  The {\em intersection lattice} of $\cA$ is the set $L(\cA)$ of all subspaces of $\bbC^n$ which can be formed by intersecting the $H_i$ partially ordered by reverse inclusion.  Note that $L(\cA)$ has a $\zh$, namely $\bbC^n$ which is the empty intersection, and is ranked by codimension.  As a Coxeter group $\fS_n=A_{n-1}$ has reflecting hyperplanes \begin{equation}
   \label{A_{n-1}} 
   \cA_{n-1}=\{ X_i=X_j \mid 1\le i< j\le n\},
\end{equation}
where $X_i$ is the $i$th coordinate function on $\bbC^n$.  Note our unusual use of capital $X$ to distinguish these functions from the variables we will use for symmetric polynomials later.
From this description it is easy to see that we have an isomorphism of posets $L(\cA_{n-1})\iso \Pi_n$.  Thus \Cref{Pi_n} can be viewed as a statement about $A_{n-1}$.  

We are finally ready to define our generalizations of the Stirling numbers.  A {\em pseudoreflection} is a linear map $\rho:\bbC^n\ra\bbC^n$ whose fixed point set is a hyperplane and which is of finite order.  A {\em complex reflection group} is a group, $G$, generated by pseudoreflections.  The group is {\em irreducible} if
the only $G$-invariant subspaces are $\bbC^n$ and the origin, and in this case $n$ is called the {\em rank} of $G$.
The finite irreducible complex reflection groups were classified by Shephard and Todd~\cite{ST:fur} into essentially one infinite family $G(m, p, n)$ and $34$ exceptional groups.
Let $\cA_G$ be the set of hyperplanes for the  pseudoreflections in $G$.  We define the {\em Stirling numbers of the second and first kinds} for $G$ by
\begin{equation}
  \label{S(G,k)}
  S(G,k) = W(L(\cA_G), n-k)
\end{equation}
and
\begin{equation}
    \label{s(G,k)}
    s(G,k) = w(L(\cA_G), n-k),
\end{equation}
respectively.

Of course, a definition is only as good as the theorems it can help prove.   In fact, one result indicating that~\eqref{S(G,k)} and~\eqref{s(G,k)} are worthy of study is already in the literature.  Let $\bx_n=\{x_1,x_2,\ldots,x_n\}$ be a set of variables.  
The corresponding {\em complete homogeneous symmetric polynomial of degree $k$}, denoted
$h_k(n)=h_k(x_1,\ldots,x_n)$, is the sum of all monomials in $\bx_n$ of degree $k$.
By contrast, the {\em elementary symmetric polynomial of degree $k$}, written $e_k(n)=e_k(x_1,\ldots,x_n)$, is the sum of all such square-free monomials.  The ordinary Stirling numbers are just specializations of these polynomials, in particular
\begin{equation}
\label{Shse}
    S(n,k)=h_{n-k}(1,2,\ldots,k) \text{ and }s(n,k)=(-1)^{n-k}e_{n-k}(1,2,\ldots,n-1).
\end{equation}

The following theorem was proved by Orlik and Solomon, although it was stated in a different form.  It shows how $G$'s Stirling numbers of the first kind can be computed using its coexponents, certain invariants of the group.
\begin{thm}[\cite{OS:urg}]
\label{OS}
If $G$ is an irreducible complex reflection group with coexponents
$e_1^*,\ldots,e_n^*$ then

\vs{10pt}

\eqqed{ s(G,k) = (-1)^{n-k} e_{n-k}(e_1^*,\ldots,e_n^*).}
\end{thm}
We will see that a similar statement can be made for $S(G,k)$ in some, but not all, irreducible complex reflection groups.

The rest of this paper is organized as follows.  
In the next section we will see how one can give combinatorial interpretations to the Stirling numbers for the groups $G(m,p,n)$ in terms of colored partitions and permutations.  This will enable us to give an even more refined connection between the M\"obius function of the intersection lattice and certain permutations indexed by partitions.
Section~\ref{iqa} will be devoted to studying an inversion-like statistic on colored partitions and permutations.  The resulting $q$-analogues satisfy generalizations of the equalities~\eqref{Shse}.  In Section~\ref{osn}, we study two ordered analogues of the $q$-Stirling numbers of the second kind.  We show that they satisfy nice alternating sum identities using both properties of symmetric polynomials and sign-reversing involutions.
Section~\ref{ca} explores a conjectured connection between the ordered Stirling numbers of the second kind and certain super coinvariant algebras.  We end with a section devoted to miscellaneous identities and an open problem.

\section{A combinatorial interpretation}
\label{ci}

It turns out that one can give a combinatorial interpretation to the $S(G,k)$ for certain complex reflection groups in terms of colored partitions.  This generalizes work of Zaslavsky~\cite{zas:grs} in types $A$, $B$, and $D$.  Using this interpretation  we will be able to derive formulas for these Stirling numbers in terms of complete homogeneous symmetric functions.  We will also show how individual M\"obius values for elements in the intersection lattice of $G$ can be interpreted in terms of colored permutations.

\subsection{Hyperplanes for \texorpdfstring{$G(m,p,n)$}{G(m,p,n)}}

The groups we will be concerned with in this section are $G(m,p,n)$, where the three parameters are positive integers with $p \mid m$, defined as follows. Let
$$
\mu_m = \{\ze\in\bbC \mid \ze^m=1\}
$$
be the $m$th roots of unity.  If $p$ divides $m$ then $G(m,p,n)$ is the set of all $n\times n$ matrices satisfying the following constraints.
\begin{enumerate}
    \item Each row and column contains exactly one non-$0$ entry.  Let $\ze_i$ be the entry in row $i$.
    \item We have $\ze_i\in\mu_m$ for all $i\in[n]$.
    \item We have $(\ze_1\ze_2\cdots\ze_n)^{m/p}=1$.
\end{enumerate}
The pseudoreflections and corresponding hyperplanes in $G(m,p,n)$ are well-known; the following is included to keep our exposition self-contained. As usual, the standard basis vectors for $\bbC^n$ are denoted
$\bfe_1,\ldots,\bfe_n$.

\begin{lem}
\label{H}
The pseudoreflections $\rho\in G(m,p,n)$ are of two forms.
\begin{enumerate}
    \item[(a)]  For some $i\in[n]$ and $\ze\in\mu_{m/p}\setm\{1\}$
$$
\text{$\rho(\bfe_i)=\ze \bfe_i$ and 
$\rho(\bfe_k)=\bfe_k$ for $k\neq i$.}
$$
    \item[(b)]  For some $i,j\in[n]$ and $\ze\in\mu_m$
$$
\text{$\rho(\bfe_i)=\ze \bfe_j$,
$\rho(\bfe_j)=\ze^{-1} \bfe_i$ and 
$\rho(\bfe_k)= \bfe_k$ for $k\neq i,j$.}
$$
\end{enumerate}
The corresponding hyperplanes are as follows.
\begin{enumerate}
    \item[(a)]  The hyperplanes $X_i=0$ for some $i\in[n]$ (provided $p<m$).
    \item[(b)]  The hyperplanes $\ze X_i=X_j$ for some $i,j\in[n]$ and $\ze\in\mu_m$ (for all $p$),
\end{enumerate}
\end{lem}
\begin{proof}
We will verify that $\rho$ is a pseudoreflection if and only if it has one of the two given forms.  Once this is done, the equations of the reflecting hyperplanes are clear.

It is easy to see that (a) and (b) yield pseudoreflections. Conversely, suppose $\rho \in G(m, p, n)$ is a pseudoreflection. Let $\pi\in\fS_n$ be the permutation whose matrix has the same nonzero positions as $\rho$.  Each $k$-cycle of $\pi$ corresponds to a $k \times k$ submatrix of $\rho$ with non-zero entries $a_1, \ldots, a_k$.  Using Laplace expansion, we find the characteristic polynomial of this submatrix is
  \[ \pm a_1 \cdots a_k - \lambda^k, \]
which has distinct roots. The characteristic polynomial of $\rho$ is the product of these and is by assumption divisible by $(1 - \lambda)^{n-1}$. Hence we must have $k \leq 2$, and there may be at most one $2$-cycle of $\pi$.

If $\pi$ is a transposition, then for some $i,j\in[n]$ there are scalars such that
$$
\text{$\rho(\bfe_i)=a \bfe_j$,
$\rho(\bfe_j)= b\bfe_i$ and 
$\rho(\bfe_k)= c_k\bfe_k$ for $k\neq i,j$.}
$$
The characteristic polynomial is then
$$
(ab-\la^2) \prod_{k\neq i,j} (c_k - \la)^{n-2}.
$$
The only way for $1-\la$ to have multiplicity $n-1$ is if $b=a^{-1}$ which also forces $c_k=1$ for all $k$, so $\rho$ is of form (b). The case when $\pi$ is the identity is similar and results in form (a). This completes the proof.
\end{proof}

\subsection{Colored partitions}

It follows from the previous lemma that the intersection lattice for $G(m,p,n)$ is the same for all $p<m$.  We will denote this lattice by
$L(m,n)$.  And the intersection lattice for $G(m,m,n)$ will be written $\Lb(m,n)$.  To give a combinatorial model for these lattices we will used colored set partitions.  
Fix $m$ and let $\zeta_m = \exp(2\pi i/m)$.
For $c\in\bbN$, we will use the abbreviation
\begin{equation}
   \label{i^c} 
   i^c = \ze_m^c \bfe_i,
\end{equation}
We will sometimes refer to $i^c$ as the {\em base $i$ with color $c$}.  Let
$$
[n^m] = \{0\} \uplus \{i^c \mid i\in[n],\ 0\le c< m\}.
$$
For any scalar $k$ and subset $S\sbe[n^m]$ we let
\begin{equation}
 \label{kS}
 kS = \{ks \mid s\in S\}.
\end{equation}
For example, if $n=2$, $m=3$ and $\ze=\ze_3$ then
$$
\ze^2 \{0,1^0, 1^2, 2^1\}
=\{\ze^2 \cdot 0,\ \ze^2\cdot\bfe_1,\ \ze^2\cdot\ze^2\bfe_1,\ \ze^2\cdot\zeta\bfe_2\}
=\{0,\ \ze^2\bfe_1,\ \ze\bfe_1,\ \bfe_2\} = \{0,1^2,1^1,2^0\}.
$$

Say that  set $T$ is a {\em multiple} of set $S$  if there is $\ze\in\mu_m$ with $T=\ze S$.  Note that this is an equivalence relation.  Say that distinct sets $S_1,S_2,\ldots, S_m$ are an {\em $m$-tuple} if $S_i$ is a multiple of $S_j$ for all $i,j\in[m]$.  This implies that, for a suitable reindexing, we have $S_i = \ze_m^{i-1} S_1$ for all $i\in[m]$.  We are now able to state our two crucial colored partition definitions, the first for $G(m,p,n)$ with $p<m$ and the second for $G(m,m,n)$.
\begin{defn}
\label{csp}
A {\em colored set partition of type $(m,n)$ } is a partition of  $[n^m]$ of the form
$$
\si = S_0/S_1/S_2/\ldots/S_{km}
$$
for some $k$ which satisfies the following two conditions:
\begin{enumerate}
    \item[(i)] $0\in S_0$ and if $i^c\in S_0$ for any $c$ then $i^d\in S_0$ for all $0\le d<m$, and
    \item[(ii)] the blocks $S_{lm+1}, S_{lm+2},\ldots, S_{(l+1)m}$ form an $m$-tuple for each $0\le l<k$.
\end{enumerate}
Block $S_0$ is called the {\em zero block} and, because of (i), it must have cardinality $lm+1$ for some $l\in[k]$.
We write $\si\ptn_{m,n} [n^m]$ and let
$$
S([n^m],k) =\{\si\ptn_{m,n} [n^m] \mid \text{$\si$ has $km+1$ blocks}\}
$$
as well as
$$
S(m,n,k)=\#S([n^m],k)
$$
Finally, we will consider the poset
$$
\Pi_{m,n} =\uplus_{k\ge 0} S([n^m],k) 
$$
ordered by refinement.

\end{defn}

\begin{defn}
A {\em colored set  partition of type $\ol{(m,n)}$} is a type $(m,n)$ partition $\si$ which also satisfies
\begin{enumerate}
    \item[(iii)] $B_0\neq\{0,i^0,i^1,\ldots,i^{m-1}\}$ for any $i$.
\end{enumerate}
We similarly write $\si\ptn_{\hs{3pt}\ol{m,n}} [n^m]$ and let $\Pib_{m,n}$ be the corresponding poset. We also use the notation
$\Sb([n^m],k)$ for the set of all $\ol{(m,n)}$ partitions with $km+1$ blocks and $\Sb(m,n,k)=\#\Sb([n^m],k)$.
\end{defn}

When writing examples, we will separate the $m$-tuples of blocks by vertical bars for clarity.  To illustrate, suppose $m=n=3$.  Then
\begin{equation}
    \label{[3^3],2}
\si = 0 \mid 1^0 2^2/1^1 2^0 / 1^2 2^1 \mid 3^0 / 3^1 / 3^2\in S([3^3],2)
\end{equation}
and $\si$ is also a type $\ol{(3,3)}$ partition.  On the other hand
\begin{equation}
    \label{[4^3],2}
\si = 0 4^0  4^1  4^2 \mid 1^0 3^2/1^1 3^0 / 1^2 3^1 \mid 2^0 / 2^1 / 2^2\in S([4^3],2)
\end{equation}
is not type $\ol{(3,4)}$.
It turns out that the two posets just defined are isomorphic to the intersection lattices for the $G(m,p,n)$.
\begin{thm}
\label{Pi_m,n}
If $p<m$ then 
$$
L(G(m,p,n)) \iso \Pi_{m,n}
$$
so that $S(G(m,p,n),k) = S(m,n,k)$ for all $k\ge0$.
Also 
$$
L(G(m,m,n))\iso\Pib_{m,n}
$$
which implies $S(G(m,m,n),k) = \Sb(m,n,k)$ for all $k\ge0$.
\end{thm}
\begin{proof}
Let $\ze=\ze_m$.
For the first statement, we will define a map $f:\Pi_{m,n}\ra L(G(m,p,n)) $ as follows.  Suppose $\si\ptn_{m,n} [n^m]$.  With each $i^j$ in the zero block associate the hyperplane $X_i=0$.  And for any pair of elements $i^c,j^d$ in a nonzero block associate the hyperplane $\ze^d X_i = \ze^c X_j$.  By \Cref{H}, these hyperplanes are all in $L(G(m,p,n))$.  Letting $f(\si)$ be the intersection of its associated hyperplanes thus gives an element of the intersection lattice.  It is easy to see that $f$ is order preserving.

To show that $f$ is an isomorphism, we construct its inverse.  Given a subspace $W$ in $L(G(m,p,n))$, let $\cA$ be the arrangement of all hyperplanes $H$ of $G(m,p,n)$ containing $W$.  We form $\si=f^{-1}(W)$ as follows.
For each hyperplane in $\cA$, if $H$ is of the form $X_i=0$ then put the elements $i^c$ in the zero block of $\si$ for all $0\le c<m$.  If $H$ has equation 
$\ze^c X_i = X_j$ then put $i^d$ in the same block with $j^{d+c}$ 
for each $0\le d<m$.  This is a type $(m,n)$ partition because the coordinate hyperplanes make sure that condition (i) is satisfied and the others guarantee condition (ii).  Again, the fact that this map is order preserving is clear.  Finally, $f$ and $f^{-1}$ are actual inverses of each other since the steps taken to do one are reversed to do the other.

The corresponding statement about the Stirling numbers is now immediate.

The construction of an isomorphism $\fb:\Pib_{m,n}\ra L(G(m,m,n))$ is similar with the following modifications. For $\fb$, if the block $S_0$ contains any colored elements it must contain at least two with different bases, say $i^c$ and $j^d$ with $i\neq j$.  For every such pair, put the hyperplane $\ze^d X_i = \ze^c X_j$ in the arrangement of associated hyperplanes.  Since all colors of $i$ and $j$ are in $S_0$, their intersection will be contained in the hyperplanes $X_i=0$ and $X_j=0$ even though these are not elements of $L(G(m,m,n))$.  Similar considerations take care of the inverse map.
\end{proof}

We can use the description of these lattices in terms of colored partitions to obtain formulas for $S(G,k)$ when $G=G(m,p,n)$ in terms of complete homogeoneous symmetric polynomials.  We will need the following well-known recursion for the $h_k(n) = h_k(x_1, \ldots, x_n)$:
\begin{equation}
    \label{h:rec}
    h_k(n) = h_k(n-1) + x_n h_{k-1}(n)
\end{equation}
for $n\ge1$, with boundary condition $h_k(0)=\de_{k,0}$ (Kronecker delta).
\begin{thm}\label{thm:S=h}
We have
\begin{equation}
   \label{S=h}
S(m,n,k) = h_{n-k}(1,m+1,2m+1,\ldots,km+1),
\end{equation}
and
\begin{equation}
    \label{S=h-nh}
    \Sb(m,n,k) = h_{n-k}(1,m+1,2m+1,\ldots,km+1) - n h_{n-k-1}(m,2m,\ldots,km).
\end{equation}

\end{thm}
\begin{proof}
For the first equality, it suffices to show that $S(m,n,k)$ satisfies the same initial conditions and recursion as the homogeneous symmetric function on the right side of~\eqref{S=h}.  Proving the former is simple, so we concentrate on the latter.  If 
$\si\in S([n^m],k)$ then let $\si'$ be the result of removing all elements of the form $n^c$ from $\si$.  There are two cases.

First suppose $n^0$  is in  a block of size one.  This implies that the same is true of $n^c$ for all $c$.  So  $\si'\in S([(n-1)^m],k-1)$. And in this case the map $\si\mapsto\si'$ is a bijection so there are $S(m,n-1,k-1)$ possible $\si$.

The other possibility is that $n^0$ is in a block of size at least two.  Now eliminating the $n^c$ leaves the same number of blocks so that $\si'\in S([(n-1)^m],k)$ for which there are $S(m,n-1,k)$ choices.  We also need to count the number of $\si$ which will result in a given $\si'$.  We can reinsert $n^0$ into any of the $km+1$ blocks of $\si'$. And once $n^0$ is placed, the two conditions for a colored partition force the blocks for the other $n^c$.  So the total number of $\si$ in this case is $(km+1)S(m,n-1,k)$.

From these two cases we get the recursion
\begin{equation}
    \label{S(m,n,k):rr}
S(m,n,k) = S(m,n-1,k-1) + (km+1) S(m,n-1,k).
\end{equation}
It is now an easy matter of reindexing to show that this agrees with equation~\eqref{h:rec} applied to $h_{n-k}(1,m+1,2m+1,\ldots,km+1)$.

Now consider~\eqref{S=h-nh}.
By definition, the set $\Sb([n^m],k)$ is just 
$S([n^m],k)$ with all $\si$ having a zero block $S_0$ of size $m+1$ removed.  So, by appealing to the first case, it suffices to show that the number of removed $\si$ is 
$n h_{n-k-1}(m,2m,\ldots,km)$.

The $m$ nonzero elements in $S_0$ are all of the form $i^c$ for some $i\in[n]$ and all values of $c$.  So there are $n$ choices for $i$.  There remains to show that the $h_{n-k-1}(m,2m,\ldots,km)$ counts the number of possibilities for the rest of the blocks.  This argument is much the same as the one given for $S(m,n,k)$ and so is left to the reader.
\end{proof}

To make the connection of the previous result with \Cref{OS} even stronger we note that, for $p<m$, the coexponents of $G(m,p,n)$ are 
\begin{equation}
 \label{Gexp}  
e_k^*=(k-1)m+1 \text{ for } k\in[n].
\end{equation}
See \cite[Table 1]{SW:hdf} for the general case.

\subsection{Colored permutations}
\label{cp}

We will now show how the individual M\"obius values in the lattices for $G(m,p,n)$, not just their rank sums, can be given a combinatorial interpretation in terms of colored permutations.  This generalizes a well-known result in type 
$A$~\cite[equation (5.9)]{Sag:aoc} as well as a recent result of Sagan and Swanson in type $B$~\cite{SS:qStB}.

Let $\fS_{[n^m]}$ be the symmetric group of permutations of $[n^m]$.
Recalling our convention~\eqref{i^c}, we see that $G(m,p,n)$ acts on $[n^m]$ by left multiplication, where we interpret $0$ as the zero vector.  Any $g\in G(m,p,n)$ thus has a corresponding $\pi_g\in\fS_{[n^m]}$, which can be decomposed into cycles.  For example, if
$$
g = \left[\barr{ccc}
0&1&0\\
i&0&0\\
0&0&1
\earr\right]\in G(4,1,3)
$$
then $g \bfe_1 = i\bfe_2$ and so $\pi_g(1^0)=2^1$.  The full cycle decomposition is
$$
\pi_g = (0)(1^0,2^1,1^1,2^2,1^2,2^3,1^3,2^0)(3^0)(3^1)(3^2)(3^3).
$$

We multiply scalars times cycles and linear sequences the same way that we do sets in~\eqref{kS}.  To illustrate, if we take the cycle $\ga=(1^0,3^2,2^3)$ in $\fS_{[3^4]}$
then
$$
i \ga = ( i\cdot\bfe_1,\ i\cdot i^2\bfe_3,\ i\cdot i^3\bfe_2)
=(1^1,3^3,2^0).
$$
If $\ga$ is a cycle in $\pi_g$ for some $g\in G(m,p,n)$ we now show that $\ga$ is of one of two types.

Suppose first that $\ga$ contains $i^c$ for some $i,c$ but no other colorings of $i$.  Then the fact that  $g$ respects scalar multiplication implies that every base $j$ in $\ga$ only occurs once.  We will call sequences of elements of $[n^m]$ with this property {\em primitive}.
Furthermore $\ze c$ is a cycle of $\pi_g$ for all $\ze\in\mu_m$.  As with colored partitions, these $m$ cycles are said to form an {\em $m$-tuple}.

The other possibility is that $\ga$ contains $i^c$ and $i^d$ for some $i$ and $c\neq d$.  In this case, $\ga$ has the form
$$
\ga = (\de,\xi\de,\xi^2\de,\ldots,\xi^{k-1}\de)
$$
where $\de$ is a primitive linear sequence of elements of $[m^n]$ and $\xi$ is a primitive $k$th root of unity for some $k>1$. We call such cycles {\em zero cycles}. 
Call a zero cycle $\ga$ {\em full} if it has the given form where $\xi=\ze_m$.
We also assume that  $(0)$, which is always a cycle of $\pi_g$, is full.
Call $\pi_g$ {\em full} if all its zero cycles are full.

We wish to relate colored permutations to colored partitions.
The {\em underlying partition} of $\pi_g$ for $g\in G(m,p,n)$ is the colored partition $\si$ obtained by replacing each of its primitive cycles by the set of elements it contains, and putting the union of the sets of all the zero cycles in the zero block.  By the previous two paragraphs, the result is indeed a colored set partition.  For example, if $g\in G(4,1,3)$ has
$$
\pi_g = (0)(1^0,1^2)(1^1,1^3)(2^0,3^2)(2^1,3^3)(2^2,3^0)(2^3,3^1)
$$
then the underlying partition is
$$
\si= 0 1^0 1^1 1^2 1^3 \mid 2^0 3^2 / 2^1 3^3 / 2^2 3^0 /2^3 3^1.
$$
In order to state the next result in a compact manner, we will use the notation
\begin{equation}
    \label{l!_m}
l!_m = l(l-m)(l-2m)\cdots r
\end{equation}
where $r \in [m]$ is the remainder of $l$ on division by $m$.
Note that this definition implies $l!_m=1$ if $l<0$ since the product is empty.
\begin{thm}
Suppose $\si=S_0/S_1/\ldots/S_{km}$ is a colored partition.
Let $b=\#S_0$ and $b_j=\# S_{jm}$ for $j\in[k]$.
If $\si\in\Pi_{m,n}$ then
\begin{align*}
  \mu(\si) &= (-1)^{n-k}  (b-m)!_m \prod_{j=1}^k (b_j - 1)!\\
  &=(-1)^{n-k}\#\{g\in G(m,1,n) \mid \text{$\pi_g$ is full and has underlying partition $\si$}\}.
\end{align*}
If $\si\in\Pib_{m,n}$ then
\begin{align*}
  \mu(\si) &= (-1)^{n-k}  (b-m-n) (b-2m)!_m \prod_{j=1}^k (b_j - 1)!\\
  &=(-1)^{n-k}\#\{g\in G(m,m,n) \mid \text{$\pi_g$ is full and has underlying partition $\si$}\}.
\end{align*}
\end{thm}
\begin{proof}
We start by considering the case when $\si\in\Pi_{m,n}$.  To prove the first equality, note that when constructing an element below $\si$ in $\Pi_{m,n}$ each of the blocks $S_{jm}$ for $0\le j\le m$ can be refined independently.  And once refinements for these blocks have been chosen, the refinements of the other blocks are forced.    The possible refinements for an $S_{jm}$ with $j\ge1$ give a poset isomorphic to the ordinary partition lattice $\Pi_{b_j}$.  Furthermore, the poset of possible refinements for $S_0$ is isomorphic to $\Pi_{m,(b-1)/m}$.  It follows that we have an isomorphism
$$
[\zh,\si] \iso \Pi_{m,(b-1)/m}\times\Pi_{b_1}\times \Pi_{b_2} \times \cdots\times \Pi_{b_k}
$$
where $[\zh,\si]$ is the closed interval between $\zh$ and $\si$.  If $P$ is a poset with a $\zh$ and a $\oh$ (unique maximal element) we write $\mu(P)=\mu(\oh)$.  Using the previous displayed equation, the fact that $\mu$ is multiplicative, \Cref{OS}, and~\eqref{Gexp} gives
\begin{align*}
 \mu(\si) &=   \mu(\Pi_{m,(b-1)/m}) \prod_{j=1}^k \mu(\Pi_{b_j})\\ 
& = (-1)^{(b-1)/m} (b-m)!_m \prod_{j=1}^k (-1)^{b_j-1} (b_j-1)!\\
\end{align*}
which is the desired expression since $(b-1)/m+\sum_{j} b_j = n$. (Incidentally, we may avoid dependence on \Cref{OS} by instead inducting on $n$.)

For the second equality, consider the number of ways of turning $\si$ into a full $\pi_g$ having it as underlying partition.  Once cycles have been put on the $S_{jm}$ for $j\ge1$, the cycles on the other nonzero blocks are forced.  And it is well-known that the number of ways to make a cycle out of $b_j$ elements is $(b_j-1)!$.  So it suffices to show that the number of ways to decompose $S_0$ into full zero cycles is $(b-m)!_m$.  Let $b=lm+1$.  We will prove the claim by induction on $l$, where we skip the base case.  Let $\rho$ be a product of full zero cycles with underlying partition $S_0$, and let $p$ be the maximum base in $\rho$.  Consider $\rho'$, the result of removing all copies of $p$ from $\rho$.   By  induction, the number of possible $\rho'$ is $(b-2m)!_m$.  To see how many $\rho$ give rise to a given $\rho'$ there are two cases.  One is that all the copies of $p$ were in a zero cycle with no other bases.  Then the full condition implies there is only one such cycle, namely $(p^0,p^1,\ldots,p^{m-1})$.  In the second case, the copies of $p$ are in a cycle of $\rho$ with other bases.  In this case, one can insert $p^0$ after any of the $b-m-1$ elements of cycles of $\rho'$ other than the cycle $(0)$.  Then the placement of the other copies of $p$ are determined by the form of zero cycles.  So the total number of possible $\rho$ is
$$
(1+b-m-1)(b-2m)!_m = (b-m)!_m
$$
as desired.

The proof $\si\in\Pib_{m,n}$ is very similar to the one just given.  In proving the second equality, the factor $b-m-n$ in the product which replaces $b-m$ from the previous case comes from removing the possible full cycle decompositions which have a cycle of the form $(p^0,p^1,\ldots,p^{m-1})$ for $p\in[n]$ which is not permitted in this lattice just as we did in proving equation~\eqref{S=h-nh}.  So we have left the details to the reader.
\end{proof}

Using the previous result we may  extend the part of Theorem~\ref{thm:Ww-Ss} 
interpreting the Stirling numbers of the first kind in terms of permutations.
to all $G(m, p, n)$. Let
\begin{align*}
  c(m, n, k)
    &= \{g \in G(m, 1, n) : \text{$\pi_g$ is full with $k$ $m$-tuples of cycles}\}
\end{align*}
and let
\begin{align*}
  \overline{c}(m, n, k)
    &= \{g \in G(m, m, n) : \text{$\pi_g$ is full with $k$ $m$-tuples of cycles}\}.
\end{align*}

\begin{cor}
  We have

 \eqqed{s(G(m, p, n), k) = (-1)^{n-k} \cdot
       \begin{cases}
          \#c(m, n, k) & \text{if $p < m$,} \\
          \#\overline{c}(m, n, k) & \text{if $p=m$}.
       \end{cases} 
}
\end{cor}

\section{Inversions and a \texorpdfstring{$q$}{q}-analogue}
\label{iqa}

We now introduce a $q$-analogue of the Stirling numbers of the second kind for $G(m,p,n)$ using an inversion-like statistic on colored set partition.  In type $A$, such $q$-analogues and statistics have a substantial history which the reader will find in~\cite{SS:qStB}.  For type $B$, they have only been studied in a few  papers~\cite{ars:cim,BGK:qrS,KAAM:cqm,SS:qStB,SW:hdf,ZD:sii}.

We will need to write our partitions in a standard form.  Define the {\em minimum base} of a nonempty subset $S\sbe[n^m]$ to be
$$
\minb S = \min\{i \mid \text{$i^c \in S$ for some color $c$}\}
$$
if $0\not\in S$, or $\minb S = 0$ if $0\in S$. 
If $\si=S_0/\ldots/S_{km}\in S([n^m],k)$ then we let
$$
s_i = \minb S_i
$$
for $0\le i \le km$.
For example, in the partition in~\eqref{[3^3],2} we have
$s_0 = 0$, $s_1=s_2=s_3=1$, and $s_4=s_5=s_6=3$.
Since $\minb S$ is an integer, it can be compared to other integers in the usual order on $\bbZ$.  Also note that if $S_1,\ldots,S_m$ is an $m$-tuple then 
$\minb S_1 = \ldots = \minb S_m$.  
 Say that $\si$ 
is in {\em standard form}  if its blocks are listed in such a way that
\begin{enumerate}
    \item $0=s_0<s_m<s_{2m}<\ldots<s_{km}$, and
    \item for $j\in[km]$ we have $s_j^r\in S_j$ where $r$ is the remainder of $j$ on division by $m$.
\end{enumerate}
It is easy to see that these two conditions imply that the standard form of $\si$ exists and is unique.  To illustrate, the standard form of the partition in~\eqref{[4^3],2} is
\begin{equation}
    \label{std}
\si = 0 4^0  4^1  4^2 \mid 1^1 3^0/1^2 3^1 / 1^0 3^2 \mid 2^1 / 2^2 / 2^0.
\end{equation}

The statistic we will use on colored partitions is as follows.  An {\em inversion} of $\si=S_0/\ldots/S_{km}$ in standard form is a pair $(i^0,S_l)$ such that
\begin{enumerate}
    \item $i^0\in S_j$ where $j<l$,
    \item $i\ge s_l$.
\end{enumerate}
Note that the first entry of an inversion pair must have color zero.  Also, because of the standard form, it is not possible to have 
$i= s_l$  although equality will become possible when we define $q$-analogues of ordered partitions.  We let $\Inv \si$ be the set of inversions of $\si$ and $\inv\si=\#\Inv\si$.
Using the partition above as an example, we have
$$
\Inv\si = \{(4^0,S_j) \mid 1\le j\le 6\} \uplus \{(3^0,S_j) \mid 2\le j\le 6\}
$$
so $\inv\si=11$.

Define the {\em $q$-Stirling numbers of the second kind} $S[m,n,k]$ to be
$$
S[m,n,k]= \sum_{\si\in S([n^m],k)} q^{\inv\si}
$$
and
$$
\Sb[m,n,k]= \sum_{\si\in \Sb([n^m],k)} q^{\inv\si}.
$$
The polynomials $S[m,n,k]$ satisfy a $q$-analogue of the recursion~\eqref{S(m,n,k):rr} for the $S(m,n,k)$.  We will need the standard $q$-analogue of $k\in\bbZ$ which is
$$
[k]_q = 1 + q + q^2 +\cdots+q^{k-1},
$$
Note that this forces $[k]_q=0$ if $k<0$ since the sum is empty.  Also, we will often drop the subscript $q$ if no confusion will result with the set $[k]$.
\begin{prop}
\label{S[m,n,k]:rec}
For $n\in\bbN$ and $k\in\bbZ$ we have
$S[m,0,k]=\de_{k,0}$ and, for $n\ge1$,
$$
S[m,n,k] = S[m,n-1,k-1]+[km+1]_q\ S[m,n-1,k].
$$
\end{prop}
\begin{proof}
We follow the same line of reasoning as in the proof of~\eqref{S(m,n,k):rr}
when $q=1$ except keeping track of inversions.
So consider $\si=S_0/\ldots/S_{km+1}\in S([m^n],k)$ and remove all copies of $n$ to obtain $\si'$.

In the case when $n^0$ is in a block of size one then, by the standard form conditions, $n^0\in S_{km+1}$.  It follows that $\inv\si'=\inv\si$ and so such partitions contribute
$S[m,n-1.k-1]$ to the polynomial for $\si$.

When $n^0$ is in a block of size at least two it will be the only copy of $n$ causing new inversions.  So
$$
\inv\si = \inv\si' + j
$$
if $n^0\in S_{km+1-j}$ where $0\le j\le km$.  It follows that the contribution of such $\si$ is
$$
S[m,n-1,k]\cdot\sum_{j=0}^{km} q^j = [km+1]\ S[m,n-1,k].
$$
Adding together the two cases completes the proof.
\end{proof}

We may now give $q$-analogues of \eqref{S=h} and \eqref{S=h-nh}.

\begin{thm}
\label{S[m,n,k]:h}
We have
\begin{equation}
   \label{S=h.q}
S[m,n,k] = h_{n-k}([1],[m+1],[2m+1],\ldots,[km+1]),
\end{equation}
and
\begin{equation}
    \label{S=h-nh.q}
    \Sb[m,n,k] = h_{n-k}([1],[m+1],[2m+1],\ldots,[km+1]) - [n]_{q^m} h_{n-k-1}([m],[2m],\ldots,[km]).
\end{equation}
\end{thm}

\begin{proof}
    Comparison of Proposition~\ref{S[m,n,k]:rec} with the complete homogeneous recursion immediately gives \eqref{S=h.q}. For \eqref{S=h-nh.q}, from \eqref{S=h.q}, it suffices to show that the $\inv$ generating function on the subset of $S([n^m], k)$ whose $0$-block is of the form $\{0, i^0, i^1, \ldots, i^{m-1}\}$ for some $i \in [n]$ is
      \[ [n]_{q^m} h_{n-k-1}([m],[2m],\ldots,[km]). \]
    
    Consider deleting all $i^c$ from such a partition $\si$ and shifting all bases $j$ with $j>i$ down one, resulting in $\si'$. Every element of $S([(n-1)^m], k)$ with zero block of size $0$ is realized $n$ times. Since $i^0$ results in $(i-1)m$ inversions, we have
      \[ \sum_\si q^{\inv\si} = \sum_{i=1}^n q^{(i-1)m + \inv\si'} = [n]_{q^m} q^{\inv\si'} \]
    where the sum is over all $\si$ resulting in $\si'$.
    
    The argument that the generating function over $\si'$ satisfies the appropriate recursion for $h_{n-k-1}([m],[2m],\ldots,[km])$ is similar to the proof of Proposition~\ref{S[m,n,k]:rec} and is again omitted.
\end{proof}

\section{Ordered Stirling numbers}
\label{osn}

In this section we will study three ordered analogues of the Stirling numbers of the second kind.  In particular, we will prove that for $G(m,1,n)$ they satisfy an alternating sum identity.  Our proofs are both algebraic and using  sign-reversing involutions.

\subsection{Lattice ordered Stirling numbers}

In type $A$, the {\em ordered Stirling numbers of the second kind} are
$$
S^o(n,k) = k!\ S(n,k).
$$
So $S^o(n,k)$ counts the number of ways to partition $[n]$ into $k$ blocks and then order the blocks (as opposed to ordering the elements within each block).  The following alternating sum identity is well known~\cite[Theorem 2.2.2]{Sag:aoc}
\begin{equation}
    \label{S^o:alt}
    \sum_{k\ge0} (-1)^{n-k} S^o(n,k) = 1.
\end{equation}
One can prove this equation either algebraically or using a sign-reversing involution. 

A $q$-analogue of~\eqref{S^o:alt} is as follows.  Let $S[n,k]=S[1,n,k]$ and
$$
S^o[n,k] = [k]!\ S[n,k]
$$
where 
$$
[k]!=[1][2]\cdots[k].
$$
Then one can show~\cite[Theorem 5.5]{SS:qStB}
\begin{equation}
    \label{S^o[1]:alt}
    \sum_{k\ge0} (-q)^{n-k} S^o[n,k] = 1.
\end{equation}

In order to find a corresponding ordered analogue for $S[m,n ,k]$, we will need an identity of Sagan and Swanson.  Given a set of variables $\bx=\{x_1,\ldots,x_n\}$ and a variable $t$ we define the {\em $\bx$-falling factorial} to be
$$
t\da_k^\bx = (t-x_1) (t-x_2)\ldots (t-x_k).
$$
\begin{thm}[\cite{SS:qStB}]
\label{h:fall}
For $n\in\bbN$

\vs{10pt}

\eqqed{t^n = \sum_{k\ge 0} h_{n-k}(k+1)\ t\da_k^\bx.}
\end{thm}

Comparing Theorems~\ref{S[m,n,k]:h} and~\ref{h:fall}, we see that to make the complete homogeneous symmetric functions agree we must let
$x_k=q\ [(k-1)m+1]$ for all $k$.  
Note that the factor of $q$ is to give the power $q^{n-k}$ as in~\eqref{S^o[1]:alt}.
Furthermore, to get the power $(-1)^{n-k}$, we must set $t=-1$.  With these substitutions we obtain
$$
t- x_k = -1- q\ [(k-1)m+1] = -[(k-1)m + 2].
$$
Thus we  let the {\em lattice ordered $q$-Stirling numbers of the second kind} be
$$
S^o[m,n,k] = [(k-1)m+2]!_m\ S[m,n,k]
$$
where $[l]!_m$ is defined analogously to~\eqref{l!_m}, replacing each factor by the corresponding polynomial.  
The term ``lattice" here comes from the fact that the unordered version comes from the intersection lattice of $G(m,1,n)$.
We have proved the following generalization of~\eqref{S^o[1]:alt}.
\begin{thm}
\label{S^o[m]:alt}
For $n\in\bbN$ and $m\ge 1$

\vs{10pt}

\eqqed{\sum_{k\ge0} (-q)^{n-k} S^o[m,n,k] = 1.}
\end{thm}

\subsection{Super ordered Stirling numbers}

There is another variety of ordered Stirling numbers which turns up in the Hilbert series of certain super coinvariant algebras.  We will have more to say about these algebraic aspects in the next section.  Here we will show how they satisfy an alternating sum identity using combinatorial means.

Define the {\em super $q$-Stirling numbers of the second kind} as
$$
\St[m,n,k] = h_{n-k}([m-1],[2m-1],\ldots,[(k+1)m-1]).
$$
It follows easily from recursion~\eqref{h:rec} that for $n\ge1$ we have
\begin{equation}
    \label{St:rec}
\St[m,n,k] = \St[m,n-1,k-1] + [(k+1)m-1] \St[m,n-1,k]
\end{equation}
and $\St[m,0,k]=\de_{k,0}$.

In order to develop a combinatorial model for the $\St[m,n,k]$ we must modify the rules of colored set partitions in \Cref{csp}.
A {\em super set partition} $\si=S_0/\ldots/S_{km}$ is obtained from a colored set partition by ordering some of the elements of $S_0$ as follows.
For each base $i$ in $S_0$ we totally order the colorings of  this base as 
$$
i^c i^{c+1} i^{c+2} \ldots i^{c+m}
$$
where the colors are taken modulo $m$ and $c\in[m-1]$.  Note that we do not permit $c=0$ (equivalently, $c=m$), so there are $m-1$ possible orderings for a given base.  For example, if $m=3$ then the possible orders are
$$
\text{$i^1 i^2 i^0$ and $i^2 i^0 i^1$}
$$
but not $i^0 i^1 i^2$.  Different bases are ordered independently with no implied ordering between separate bases.  In examples, we will put spaces between the various base orderings.  For example, if $m=3$, $n=3$, and $k=1$ then
$$
\text{$0\ 1^1 1^2 1^0\ 3^2 3^0 3^1 \mid 2^1/ 2^2/2^0$
and
$0\ 1^2 1^0 1^1\ 3^2 3^0 3^1 \mid 2^1/ 2^2/2^0$}
$$
are two different super colored set partitions.  On the other hand
$$
0\ 1^1 1^2 1^0\ 3^2 3^0 3^1 \mid 2^2/ 2^0/2^1
$$
is the same as the first super colored set partition above since the order of the blocks $S_1.\ldots,S_{km}$ is immaterial.  Let $\St([n^m],k)$ be the set of super colored set partitions of $[n^m]$ with $km+1$ blocks.

We will also need to enhance the definition of inversion in the super case.  
The definition of standard form for super colored set partitions is the same as the one for their ordinary cousins.
If $\si=S_0/\ldots/S_{km}$ is a super colored set partition in standard form  then the {\em inversions} of $\si$ are of two types:
\begin{enumerate}
    \item pairs $(i^0,S_l)$ which are inversions of the corresponding ordinary colored set partition, and
    \item pairs $(i^0,i^c)$ where $i^0,i^c\in S_0$ and $i^c$ is to the right of $i^0$ is the ordering on base $i$.
\end{enumerate}
We let $\Inv\si$ be the set of inversions of $\si$ and $\inv\si=\#\Inv\si$.
To illustrate, if
$$
\si = 0\ 1^2 1^0 1^1\ 3^2 3^0 3^1 \mid 2^1/ 2^2/2^0
$$
then
$$
\Inv\si = \{(1^0,1^1),\ (3^0,3^1)\}\ \uplus\
\{(3^0,S_1),\ (3^0,S_2),\ (3^0,S_3)\}
$$
so $\inv\si =5$.  The following result should not be unexpected.
\begin{thm}
We have
$$
\St[m,n,k] = \sum_{\si\in \St([n^m],k)} q^{\inv\si}.
$$
\end{thm}
\begin{proof}
 It suffices to show that the sum satisfies the same boundary condition and recursion~\eqref{St:rec} as $\St[m,n,k]$.  This proof follows the same line of reasoning as that of Proposition~\ref{S[m,n,k]:rec} so we will only sketch the differences.  We keep the same notation as in that demonstration.
 
 If deleting copies of $n$ results in $\si'$ having fewer blocks than $\si$, then this base was not in $S_0$.  So exactly the same argument gives the first term of the recursion.  Now assume the $n$'s were in blocks with other elements.  The contribution of the  $\si'$ is $\St[m,n-1,k]$ as before.  So consider the inversions caused by $n^0$. If it was in one of the nonzero blocks then, as previously, these contribute $[km]$.  If $n^0\in S_0$ then it is in inversion with all the nonzero blocks as well as any $n^c$ coming after it in the ordering of this base.  This gives an extra $q^{km}+q^{km+1}+\cdots+q^{(k+1)m-2}$.
 Adding the two contributions results in the desired factor of $[(k+1)m-1]$.
 \end{proof}

 The {\em ordered super $q$-Stirling numbers of the second kind} are
 $$
 \St^o[m,n,k]=[km]!_m\ \St[m,n,k].
 $$
An {\em ordered super set partition} is a sequence
$\om=(S_0/S_1/\ldots/S_{km})$ such that $\si=S_0/S_1/\ldots/S_{km}$
is a super set partition satisfying
\begin{equation}
 \label{S_{i+1}}   
 S_{i+1} = \ze_m S_i
\end{equation}
for all $i\ge1$ not divisible by $m$.
In this case we write $\om\comp [n^m]$.
Note the use of parentheses to indicate that the order of the blocks is important.  For example, when $m=n=3$ and $k=2$
\begin{equation}
    \label{sup:ptn}
    \text{$(0\  3^2 3^0 3^1 \mid 1^0 / 1^1 /1^2 \mid 2^1/ 2^2/ 2^0)$
and
$(0\  3^2 3^0 3^1 \mid 1^2/ 1^0 / 1^1 \mid 2^1/2^2/2^0)$}
\end{equation}
are two different ordered super partitions, while
$$
(0\  3^2 3^0 3^1 \mid 1^1 / 1^0 /1^2 \mid 2^1/ 2^2/ 2^0)
$$
is not an ordered super partition because 
$S_2=\{1^0\} \neq \ze_3 S_1$.
We let $\St^o([n^m],k)$ be the set of ordered super partitions of $[n^m]$ into $km+1$ blocks.  The definition of the inversion statistic is exactly the same as in the unordered case.  The proof of the next result is so similar to that of the previous one that we omit it.
\begin{thm}
\label{St^o:sum}
We have

\vs{10pt}

\eqqed{\St^o[m,n,k] = \sum_{\om\in \St^o([n^m],k)} q^{\inv\om}.}
\end{thm}

We will now show that the $\St^0[m,n,k]$ satisfy an alternating sum identity.  We could do this algebraically through the use of \Cref{h:fall}.  But we will choose to give a combinatorial proof via a sign-reversing involution
$\io$ on the $\om\comp[n^m]$.  This function will be composed of two maps, one which splits blocks and one which merges them.  In order to define them we will use the {\em maximum base} of a nonempty $S\sbe[n^m]$ which is
$$
\maxb S = \max\{i \mid \text{$i^c\in S$ for some color $c$}\}.
$$
Also, as usual, all colors are to be taken modulo $m$.
\begin{defn}
\label{si:mu}
Suppose we have an ordered super partition $\om=(S_0/\ldots/S_{km})$ and consider $M=\maxb S_{jm}$ for some $j$.
Say that $\om$ is {\em splittable at $M$} if $\#S_{jm}\ge2$.
We define the {\em splitting map $\si_M$} using cases depending on whether $j=0$ or $j>0$.  If $j=0$ then 
$S_0$ contains the sequence $M^c M^{c+1} \ldots M^{c+m-1}$ where $c\neq0$.  In this case, we 
{\em split right} by letting
$$
\si_M(\om) 
= (S_0'/M^{c+1}/M^{c+2}/\ldots/M^{c+m}/S_1/S_2/\ldots/S_{km})
$$
where $S_0'$ is $S_0$ with all colors of $M$ removed.
If $j>0$ then write
\begin{equation}
    \label{S_i'}
S_{(j-1)m+1} = S_{(j-1)m+1}'\uplus \{M^c\},\
\ldots,\
S_{jm} = S_{jm}' \uplus \{M^{c+m-1}\}.
\end{equation}
There are now two more cases.  If $c\neq 0$, then we split right again where
$$
\si_M(\om) = 
(\ldots/S_{(j-1)m+1}'/S_{(j-1)m+2}'/\ldots/S_{jm}'/M^{c+1}/M^{c+2}/\ldots/M^{c+m}/\ldots)
$$
and the blocks before $S_{(j-1)m+1}'$ and after $M^{c+m+1}$ have been left invariant.  If $c=0$ then, in this third case, we {\em split left}
$$
\si_M(\om) = 
(\ldots/M^1/M^2/\ldots/M^{m-1}/M^0/S_{(j-1)m+1}'/S_{(j-1)m+2}'/\ldots/S_{jm}'/\ldots)
$$
with the same conventions as in the second case about invariant blocks.

The inverse of the splitting map will be a {\em merging map}, $\mu_M$, which unions blocks together.  Let us keep  the notation of the previous paragraph, and suppose that we have $\#S_{jm}=1$.  This implies that
\begin{equation}
\label{mer}
    S_{(j-1)m+1}=\{M^c\},\ S_{(j-1)m+2}=\{M^{c+1}\},\ \ldots,\ S_{jm}=\{M^{c+m-1}\}
\end{equation}
for some $c$.  The cases when $\om$ is {\em mergeable at $M$} depend on the value of $c$.
Suppose first that $c\neq 1$ and $M>\maxb S_{(j-1)m}$.  This gives us two cases depending on whether we are merging into the zero block or not.  If $j=1$ then we {\em merge left} by defining
\begin{equation}
  \label{mer:0}  
  \mu_m(\om) = (S_0''/S_{m+1}/S_{m+2}/\ldots/S_{km})
\end{equation}
where $S_0''$ is $S_0$ with the sequence $M^{c-1} M^c\ldots M^{c+m-2}$ added.
In the case $j>1$,  we merge left again producing
$$
\mu_M(\om)=
(\ldots/S_{(j-2)m+1}\uplus\{M^{c-1}\}/S_{(j-2)m+2}\uplus\{M^c\}/\ldots/S_{(j-1)m}\uplus\{M^{c+m-2}\}/\ldots).
$$
The third case is when $c=1$ and $M>\maxb S_{(j+1)m}$.  Then we {\em merge right} by
$$
\mu_M(\om)=
(\ldots/S_{jm+1}\uplus\{M^0\}/S_{jm+2}\uplus\{M^1\}/\ldots/S_{(j+1)m}\uplus\{M^{m-1}\}/\ldots).
$$
\end{defn}

\begin{exa}
Here are some examples when $n=m=3$ of the splitting map (forward arrows) and merging map (reverse arrows) illustrating each of the cases above when $M=3$.
\begin{align*}
(0\ 1^2 1^0 1^1\ 3^1 3^2 3^0 \mid 2^0/ 2^1/ 2^2)
&\longleftrightarrow (0\ 1^2 1^0 1^1 \mid 3^2/ 3^0/ 3^1 \mid 2^0/ 2^1/ 2^2),\\
(0\ 1^2 1^0 1^1 \mid 2^0 3^1/ 2^1 3^2/ 2^2 3^0)
&\longleftrightarrow (0\ 1^2 1^0 1^1 \mid 2^0 / 2^1 / 2^2  \mid 3^2/ 3^0/ 3^1),\\
(0\ 1^2 1^0 1^1 \mid 2^0 3^0/ 2^1 3^1/ 2^2 3^2)
&\longleftrightarrow (0\ 1^2 1^0 1^1 \mid 3^1/3^2/3^0 \mid 2^0/ 2^1/ 2^2).
\end{align*}
\end{exa}

Note that from the definition just given, if $\om$ is splittable at $M$ then $\si_M(\om)$ is mergeable at $M$.   Also  $\mu_M\si_M(\om)=\om$ because the value of $c$ determines whether these operations take place to the left or the right, and one is not permitted the sequence
$i^0 i^1\ldots i^{m-1}$ in $S_0$.  If one reverses the roles of $\si_M$ and $\mu_M$ then the same is true.  Merging and splitting maps have found applications in a number of areas, including computing cancellation-free formulas for antipodes in Hopf algebras as shown by Benedetti and Sagan~\cite{BS:ai}.  We can now define our involution.

\begin{defn}
\label{io:def}
Define $\io:\uplus_k \St([n^m],k)\ra\uplus_k \St([n^m],k)$ as follows.
Given $\om=(S_0/\ldots/S_{km})$, let $M$ be the largest integer (if any) such that $M=\maxb S_{jm}$ for some $j$ and $\om$ is either splittable or mergeable at $M$.  Note that by the  cardinality restriction, $\om$ cannot be both.   Let
$$
\io(\om)=
\begin{cases}
\si_M(\om)  &\text{if $\om$ is splittable at $M$,}\\
\mu_M(\om)  &\text{if $\om$ is mergeable at $M$,}\\
\om         &\text{if no such $M$ exists.}
\end{cases}
$$
Note that $\om$ is an involution because of the remarks from the previous paragraph and the fact that the largest $M$ such that $\om$ is splittable or mergeable is preserved by the maps.
\end{defn}

\begin{thm}
\label{St^o:alt}
We have
$$
\sum_{k\ge0} (-q)^{n-k} \St^o[m,n,k]=1.
$$
\end{thm}
\begin{proof}
Throughout the proof we keep the notation of \Cref{si:mu}.
If $\om\in\St^o([n^m],k)$ then define its {\em sign} to be
\begin{equation}
\label{sgn}
   \sgn\om=(-1)^{n-k}
\end{equation}
Note that both $\si_M$ and $\mu_M$ are sign-reversing since they change the number of $m$-tuples in $\om$ by one.

By the previous definition and  \Cref{St^o:sum}
we have
\begin{equation}
\label{sgn:sum}
\sum_{k\ge0} (-q)^{n-k} \St^o[m,n,k] 
= \sum_{ \om\in \uplus_k \St^o([n^m],k)}\sgn\om\ q^{n-k+\inv\om}.
\end{equation}

We claim that in any two-cycle of $\io$, the terms for $\om$ and $\io(\om)$ will cancel because they have opposite sign, but the same value of $n-k+\inv$.  In particular, applying $\si_M$ will increase both $k$ and  $\inv$ by one.  Since $\mu_M$ is the inverse, it will do the opposite.  So in either case the exponent on $q$ is preserved.

We will only check the claim for the first case in the definition of $\si_M$ and leave the other two to the reader.  The fact that $k$ increases by one has already been noted.  As far as the number of inversions, the only inversions which change are those caused by $M^0$.
It is in inversion with the elements after it in its sequence in $S_0$ as well as with certain $S_j$ for $j\ge1$.  The latter inversions are carried over in $\si_M(\om)$.  However, $M^0$ moves one place to the left in the blocks $S_1'/\ldots/S_m'$ of $\om'=\si_M(\om)$ thus causing one more inversion.

To finish the proof, it suffices to show that $\io$ has a unique fixed point
$$
\om_0= (0 \mid 1^1/1^2/\ldots/1^{m-1}/1^0 \mid  \ldots \mid n^1/n^2/\ldots/n^{m-1}/n^0 )
$$
since 
$$
\sgn\om_0\ q^{n-k-\inv\om} =(-1)^{n-n}\ q^{n-n-0} = 1.
$$
If $\io(\om)=\om$ then $\om$ can not have any blocks of size at least $2$ since otherwise it would be splittable.  Next, we claim that the $m$-tuples of singletons must be arranged in increasing order of base.  If not, then there would have to be an index $j\ge1$ with $\maxb S_{(j-1)m}<\maxb S_{jm}>\maxb S_{(j+1)m}$.  But in this case the $m$-tuple containing $S_{jm}$ could be merged either to the left or right depending on the value of $c$.  Finally, we have to rule out the possibility that, for some $j\in[n]$, the $j$th $m$-tuple begins with $j^c$ for $c\neq1$.  But if this were true, then that $m$-tuple could be merged left. So $\io_0$ is the only possible fixed point and it is easy to check that it is indeed left invariant by $\io$.  
\end{proof}

\subsection{Chan-Rhoades Stirling numbers}

These Stirling numbers appear in a paper of Chan and Rhoades~\cite{CR:Rnk} in studying coinvariant algebras for $G(m,1,n)$.  The unordered version is the same as the $S[m,n,k]$.  But the factors in the product are one larger than the exponents rather than the coexponents, and thus equal the degrees.

Define the {\em Chan-Rhoades ordered $q$-Stirling numbers} to be
\begin{equation}
\label{CR:def}
S_{\CR}^o[m,n,k] = [km]!_m\ S[m,n,k].
\end{equation}
A {\em  CR-ordered set partition} is a sequence $\om=(S_0/S_1/\ldots/S_{km})$ satisfying the same restriction~\eqref{S_{i+1}} on the nonzero blocks, but with no ordering of the elements of the zero block.  
We let $S_{\CR}^o([n^m],k)$ be the set of such partitions.
For example, the partitions in~\eqref{sup:ptn} are still distinct as CR partitions, and the CR partition
$$
(0\  3^0 3^1 3^2 \mid 1^0 / 1^1 /1^2 \mid 2^1/ 2^2/ 2^0)
$$
is the same as the first of that pair.

The definition of inversion is exactly the same as in the unordered case.  The reader should not find it hard to supply a proof of the following proposition.
\begin{thm}
We have

\vs{10pt}

\eqqed{
S_{\CR}^o[m,n,k]= \sum_{\om\in S_{\CR}^o([n^m],k) } q^{\inv\om}.
}
\end{thm}

The splitting and merging maps are similar to those in the super case.

\begin{defn}
\label{CR:si:mu}
Let $\om=(S_0/\ldots/S_{km})$ be a CR-ordered set partition and $M=\maxb S_{jm}$ for some $j$.  Call $\om$
{\em splittable at $M$} if $\#S_{jm}\ge2$.  To define the {\em splitting map} $\si_m$, first consider the case when $j=0$.  In that case we {\em split right}
$$
\si_M(\om) = (S_0'/M^0/M^1/\ldots/M^{m-1}/S_1/S_2/\ldots/S_{km})
$$
where $S_0'$ is obtained from $S_0$ by removing all copies of $M$.  If $j\ge0$ then, using the notation of~\eqref{S_i'}, we  split right again if $c=1$ to get
$$
\si_M(\om) = (\ldots/S_{(j-1)m+1}'/S_{(j-1)m+2}'/\ldots/S_{jm}'/M^0/M^1/\ldots/M^{m-1}/\ldots)
$$
or {\em split left} if $c\neq 1$ and obtain
$$
\si_M(\om) = (\ldots/M^{c-1}/M^{c+m-2}/\ldots/M^{m-1}/M^0/S_{(j-1)m+1}'/S_{(j-1)m+2}'/\ldots/S_{jm}'/\ldots).
$$

We now assume that we have an $m$-tuple of blocks as in~\eqref{mer}.  The cases when $\om$ is {\em mergeable at $M$} depend on $c$.  Suppose $c=0$ and $M>\maxb S_{(j-1)m}$.  If $j=1$ then {\em merge left} exactly as in~\eqref{mer:0} except that no order is put on the copies of $M$ in $S_0''$.  If $j>1$ then merging left is defined by
$$
\mu_M(\om)=
(\ldots/S_{(j-2)m+1}\uplus\{M^{1}\}/S_{(j-2)m+2}\uplus\{M^2\}/\ldots/S_{(j-1)m}\uplus\{M^{m-1}\}/\ldots).
$$
Lastly, of $c\neq 0$ then we can {\em merge right} by
$$
\mu_M(\om)=
(\ldots/S_{jm+1}\uplus\{M^{c+1}\}/S_{jm+2}\uplus\{M^{c+2}\}/\ldots/S_{(j+1)m}\uplus\{M^{c+m}\}/\ldots).
$$
\end{defn}

\begin{exa}
\label{S_CR:ex}
As some illustrations of the CR maps, consider $n=m=3$.  In the next three examples we always have $M=3$ with splitting or merging going forward or backward, respectively.
\begin{align*}
(0\ 1^1 1^2 1^0\ 3^1 3^2 3^0 \mid 2^0/ 2^1/ 2^2)
&\longleftrightarrow (0\ 1^2 1^0 1^1 \mid 3^0/ 3^1/ 3^2 \mid 2^0/ 2^1/ 2^2),\\
(0\ 1^1 1^2 1^0 \mid 2^0 3^1/ 2^1 3^2/ 2^2 3^0)
&\longleftrightarrow (0\ 1^1 1^2 1^0 \mid 2^0 / 2^1 / 2^2  \mid 3^0/ 3^1/ 3^2),\\
(0\ 1^1 1^2 1^0 \mid 2^0 3^0/ 2^1 3^1/ 2^2 3^2)
&\longleftrightarrow (0\ 1^1 1^2 1^0 \mid 3^2/3^0/3^1 \mid 2^0/ 2^1/ 2^2).
\end{align*}
Note that
$$
(0 \mid 1^2/1^0/1^1 \mid 2^1/2^2/2^0 \mid 3^1/3^2/3^0)
$$
cannot be split since all blocks have size one.  It also cannot be merged.  Indeed, there can be no merges to the right because the bases are weakly increasing.  And one can only merge left if one has the sequence
$M^0/M^1/\ldots/M^{m-1}$.
\end{exa}

Define an involution on $\io$ on $\uplus_k S_{\CR}^o([n^m],k)$ exactly as in \Cref{io:def}.  This will give us a combinatorial proof of the next result which can also be obtained from substitution in \Cref{h:fall}.
\begin{thm}
\label{S_CR:alt}
We have
$$
\sum_{k\ge0} (-q^{m-1})^{n-k} S_{\CR}^o[m,n,k] = [m-1]^n.
$$
\end{thm}
\begin{proof}
  The proof is similar enough to that of \Cref{St^o:alt} that we will just discuss the differences.  From the previous result we have
  $$
  \sum_{k\ge0} (-q^{m-1})^{n-k} S_{\CR}^o[m,n,k] 
  =\sum_{ \om\in \uplus_k S_{\CR}^o([n^m],k)}\sgn\om\ q^{(m-1)(n-k)+\inv\om}
  $$
where $\sgn\om$ is defined as in~\eqref{sgn}.

For the cancelling terms, we wish to leave $(m-1)(n-k)+\inv\om$ invariant.  So when splitting or merging, we want $\inv$ to increase or decrease by $m-1$, respectively.  We will check the case when $\si_M$ splits to the right.  This can only happen if $M^0\in S_{jm}$ for some $j$.  The inversions of $M^0$ with other blocks to its right are preserved by the split.  And in the new $m$-tuple $M^0/M^1/\ldots/M^{m-1}$ we see that $M^0$ is causing $m-1$ new inversions as desired.

We claim that the fixed points of $\io$ are precisely the ordered partitions  with $k$ $m$-tuples where the $i$th $m$-tuple is of the form 
\begin{equation}
\label{m-tuple}
    i^c/i^{c+1}/\ldots/i^{c+m-1}
\end{equation} 
for some $c\neq0$.  Checking that these are fixed is similar to the argument for the last partition in \Cref{S_CR:ex}.  To see that there are no others, suppose that $\om$ is a fixed partition.  Then it must have all blocks of size one, since larger blocks can be split.  The bases must also be in weakly increasing order by the same argument as in the proof of \Cref{St^o:alt}.  Finally, if $c=0$ then the $m$-tuple can always be merged left.

To complete the proof, we must compute the terms for the fixed points $\om$.  Since $k=n$, all such $\om$ have $\sgn\om=1$.  Furthermore, we see that the $m$-tuple in~\eqref{m-tuple} contributes $q^{c-1}$ for $c=1,\ldots,m-1$.  So the total contribution of the $i$th $m$-tuple is $[m-1]$ and these contributions are independent, giving the desired product.
\end{proof}


\section{Coinvariant algebras}
\label{ca}

We now discuss various generalizations of the coinvariant algebra for $G=G(m,1,n)$.  Throughout this section, $G$ will refer to this group.  In the case of super coinvariant algebras, we will propose analogues of the Artin basis which generalize those conjectured by Sagan and Swanson in types $A$ and $B$~\cite{SS:qStB} and recently proved to be a basis in type $A$ by Angarone, Commins, Karn, Murai and Rhoades~\cite{ACKMR:sch}.  If these are indeed bases in more general type, then the graded Hilbert functions can be expressed in terms of the $\St^o[m,n,k]$.  We will also discuss another variant of the coinvariant algebra due to Chan and Rhoades~\cite{CR:Rnk}.

\subsection{Ordinary coinvariants}

Recall that $\bx=\{x_1,\ldots,x_n\}$.  Let $G$ act on the polynomial algebra $\bbC[\bx]$ in the same way as its action on $[n^m]$ as described in Subsection~\ref{cp} where $x_i$ takes the place of $\bfe_i$.  Let $\bbC[\bx]^G_+$ be the invariants of this action which have zero constant term.  The corresponding {\em coinvariant algebra} is the quotient.
$$
\RG_{m,n} = \frac{\bbC[\bx]}{\spn{\bbC[\bx]^G_+}}.
$$

There is a standard basis for $\RG_{m,n}$ which can be described as follows.  Suppose $\al=(\al_1,\ldots,\al_n)$ is a {\em weak composition}, that is, a sequence of nonnegative integers called {\em parts}.  We let
$$
\bx^\al = x_1^{\al_1}\cdots x_n^{\al_n}
$$
We also partially order weak compositions componentwise.

\begin{defn}
The {\em Artin basis} for $\RG_{m,n}$ is
$$
\cA_{m,n} = \{\bx^\al \mid \al\le (m-1,2m-1,\ldots,nm-1)\}.
$$
The composition $(m-1,2m-1,\ldots,nm-1)$ is called the {\em $(m,n)$-staircase}.  An example is given in Figure~\ref{stair} where the parts are vertical.
\end{defn}

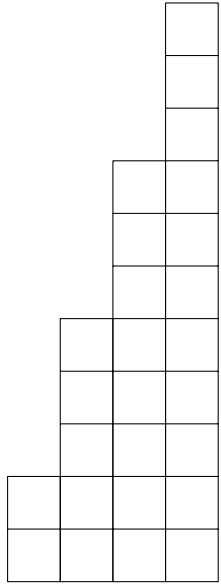
\begin{figure}
\bce
\begin{tikzpicture}[scale=.7]
\draw (0,0) grid (1,2);
\draw (1,0) grid (2,5);
\draw (2,0) grid (3,8);
\draw (3,0) grid (4,11);
\end{tikzpicture}
\ece
    \caption{The $(3,4)$-staircase $(2,5,8,11)$}
    \label{stair}
\end{figure}

The fact that $\cA_{m,n}$ is a basis for $\RG_{m,n}$ trivializes the computation of its Hilbert series. The following result is originally due to Chevalley. See \cite{SW:hdf} for its history and a proof.
\begin{thm}
The Hilbert series of $\RG_{m,n}$ is

\vs{10pt}

\eqqed{
\Hilb(\RG_{m,n},q) = [mn]!_m.
}
\end{thm}

\subsection{Super coinvariants}

\begin{figure}
\bce
\begin{tikzpicture}[scale=.7]
\draw(-2,6) node{$\be(T)=$};
\draw(-2,14) node{$\phi(T)=$};
\draw(.5,14) node{$(1,$};
\draw(3,14) node{$2,$};
\draw(5.5,14) node{$1,$};
\draw(7,14) node{$0,$};
\draw(8.5,14) node{$1)$};
\draw(0,0) grid (1,1);
\draw[fill] (0,.9) rectangle (1,1.1);
\draw (1,0) grid (2,2);
\draw (2,0) grid (4,4);
\draw[fill] (2,3.9) rectangle (4,4.1);
\draw (4,0) grid (5,5);
\draw (5,0) grid (6,7);
\draw[fill] (5,6.9) rectangle (6,7.1);
\draw (6,0) grid (7,8);
\draw (7,0) grid (8,11);
\draw[fill] (6.9,9.9) rectangle (7.1,10.1);
\draw (8,0) grid (9,13);
\draw[fill] (8,12.9) rectangle (9,13.1);
\draw(-2,-1) node{$T=$};
\draw(.5,-1) node{$\{1,$};
\draw(2.5,-1) node{$3,$};
\draw(3.5,-1) node{$4,$};
\draw(5.5,-1) node{$6,$};
\draw(8.5,-1) node{$9\}$};
\draw(-2,-2) node{$U=$};
\draw(1.5,-2) node{$\{2,$};
\draw(4.5,-2) node{$5,$};
\draw(6.5,-2) node{$7,$};
\draw(7.5,-2) node{$8\},$};
\end{tikzpicture}
\ece
    \caption{Compositions $\be(T)$ and $\phi(T)$ when $m=3$, $n=9$ and $T=\{1,3,4,6,9\}$}
    \label{be(T):fig}
\end{figure}

Consider a second set of variables $\bt=\{\th_1,\ldots,\th_n\}$ which commute with the $x$-variables and anticommute among themselves, that is
$$
    \th_i \th_j = -\th_j\th_i.
$$
So $\th_i^2=0$ for all $i$.  It follows that monomials in these variables are square-free and so indexed by subsets $T\sbe[n]$.  We will let
$$
\th_T = \prod_{i\in T} \th_i,
$$
where for concreteness we take the product in increasing order of $i$.

Let $G$ act on the algebra $\bbC[\bx,\bt]$ diagonally, where $x_i$ or $\theta_i$ takes the place of $\bfe_i$.  Keeping the notation from the previous subsection, we have  a {\em super coinvariant algebra} which is
$$
\SRG_{m,n} = \frac{\bbC[\bx,\bt]}{\spn{\bbC[\bx,\bt]^G_+}}.
$$
We conjecture that this algebra is closely related to the super ordered Stirling numbers.  We will use the variable $t$ to track the $\theta$-degrees in the next Hilbert series.
\begin{conj}
\label{SRG:con}
We have
$$
\Hilb(\SRG_{m,n};q,t) = \sum_{k\ge0} \St^o[m,n,k] t^{n-k}
$$
\end{conj}

As some evidence for this conjecture, we offer the following result.  Let
$\SRG_{m,n}^k$ be the $k$th graded piece of $\SRG_{m,n}$ in the theta grading.
\begin{prop}
We have
\begin{enumerate}
    \item[(a)] $\Hilb(\SRG_{m,n}^0)=[nm]!_m$.
    \vs{5pt}
    \item[(b)] $\Hilb(\SRG_{m,n}^n)=[m-1]^n$.
\end{enumerate}
\end{prop}
\begin{proof}  Part (a) holds for all pseudo-reflection groups in characteristic zero by Chevalley's Theorem.  For part (b), when $m=1$ we have $\Hilb(\SRG_{1,n}^n)=\de_{n,0}=[m-1]^n$.  When $m\ge2$, by the arguments in \cite[\S3]{SW:hdf},
$$
\SRG_{m,n}^n =\mathcal{H}' \theta_1 \cdots \theta_n 
$$
    where $\mathcal{H}'$ is the space of polynomials annihilated by $\partial_{x_i}^{m-1}$ for all $i \in [n]$. A basis for $\mathcal{H}'$ consists of monomials $\bx^\alpha$ where all $\alpha_i \leq m-2$. Thus the contribution of the powers of $x_i$ is $[m-1]$ for $i\in[n]$ and the result follows.
\end{proof}

We  now describe a conjectural analogue of the Artin basis which would prove Conjecture~\ref{SRG:con}.
Given $T=\{t_1<t_2<\ldots<t_{n-k}\}$ we will extend a staircase composition to a composition $\be(T)$ as follows.  Consider the complement 
$U=[n]\setm T=\{u_1,u_2,\ldots,u_k\}$.  We will construct a weak composition $\be(T)=(\be_1,\ldots,\be_n)$ where we consider $T$ and $U$ as indexing the $n$ columns of the  diagram of $\be(T)$.  See Figure~\ref{be(T):fig} for an example.
Create an $(m,k)$-staircase by putting a column of $im-1$ boxes in the column for $u_i$ where $i\in[k]$.  For the column indexed by $t_j$, let $i\in[k+1]$ be such that $u_{i-1}<t_j<u_i$ where $u_0=0$ and $u_{k+1}=n+1$.  Put a column of $im-2$ boxes in the column for $t_j$.  Since $T\uplus U=[n]$, we have described all the $n$ parts of $\be(T)$.  It will be useful to have another weak composition $\phi=(\phi_1,\ldots,\phi_{k+1})$ where
$$
\phi_i = \text{number of columns of length $im-2$ in $\be(T)$}.
$$
Intuitively, the entries of $\phi$ counts the number of columns which were added to the staircase at each of the possible heights.  \Cref{be(T):fig} displays $\phi(T)$ above the diagram of $\be(T)$.

\begin{defn}
\label{SA:def}
The {\em super Artin set} for $\SRG_{m,n}$ is
$$
\SA_{m,n}=\{ \bx^\al \th_T \mid \text{$T\sbe[n]$ and $\al\le\be(T)$}\}.
$$
\end{defn}

Here is the conjecture alluded to above.

\begin{conj}
\label{SAG:con}
The set $\SA_{m,n}$ is a basis for $\SRG_{m,n}$.
\end{conj}

\begin{prop}
  \Cref{SAG:con} implies \Cref{SRG:con}.
\end{prop}
\begin{proof}
  Fix $T\sbe[n]$ with $\#T=n-k$ and consider the monomials in $\SA_{m,n}$ whose theta-factor is $\th_T$.  The columns for $U$ form an $(m,k)$-staircase and so will contribute a $q$-factor of $[km]!_m$.  Similar reasoning and the definition of $\phi(T)$ shows that the columns for $T$ contribute a factor of
  $$
  [m-1]^{\phi_1}[2m-1]^{\phi_2}\cdots[(k+1)m-1]^{\phi_{k+1}}.
  $$
 So summing over all $T$ with cardinality $n-k$ results in a coefficient of
 $$
 [km]!_m\ h_{n-k}([m-1],[2m-1],\ldots,[(k+1)m-1]) = \St^o[m,n,k]
 $$
 for the power $t^{n-k}$ by equation ~\eqref{St:rec}.  Summing  over $k$ gives the desired Hilbert series.
\end{proof}

There is a way to index $\SA_{m,n}$ using  ordered  super set partitions.  Suppose that we are given $\om=(S_0/\ldots/S_{km})\comp [n^m]$ and $s\in[n]$.  We let
$$
\inv_s \om = \text{number of pairs in $\Inv\om$ with $s^0$ as first component}.
$$
These are the parts of the {\em inversion composition}
\beq
\label{eta}
\eta(\om)=(\inv_1\om,\inv_2\om,\ldots,\inv_n\om).
\eeq
For example, suppose
$$
\om = (0\ 4^2  4^0  4^1 \mid 1^0 3^2/1^1 3^0 / 1^2 3^1 \mid 2^0 / 2^1 / 2^2).
$$
Now $\inv_3 \om=4$ because of $(3^0,S_i)$ for $3\le i \le 6$.  Also $\inv_4\om=7$ caused by $(4^0,4^1)$ and $(4^0,S_i)$ for $i\in[6]$.  And $\inv_s \om=0$ for all other $s\in[4]$ so that
$$
\eta(\om)=(0,0,4,7).
$$
We also need the set
$$
T(\om) = 
\{ t\in[n] \mid \text{$t>\minb S_i$ where $S_i$ is the block containing $t^0$}\}.
$$
In the previous example, $T(\om)=\{3,4\}$ because $3^0$ is in a block with $1^1$, and $4^0$ is in a block with $0$.
\begin{prop}
  We have
  $$
  \SA_{m,n} = \{ \bx^{\eta(\om)} \th_{T(\om)} \mid \om\comp[n^m]\}.
  $$
\end{prop}
\bprf
It suffices to define a weight-preserving bijection  from pairs $(T,\al)$ appearing in~\Cref{SA:def} to the $\om\comp[n^m]$.  Given $(T,\al)$ we construct $\om$ inductively as follows.  We start with $\om=(0)$ and insert all colorings of  $1,2,\ldots,n$ in order into blocks according to the following rules when it comes to inserting the colors of $k$.  An example follows this demonstration.  Note that for blocks $S_i$ with $i\ge1$ it suffices to specify the block containing $k^0$ since then the other colors are determined by~\eqref{S_{i+1}}.  And for the zero block, one just needs to specify the position of $k^0$ in its sequence to determine the sequence.
\ben
\item  If $k\in T$ then put $k^0$ in the existing block of $\om$, and in the correct position in its sequence if using the zero block, so that  exactly $\al_k$ new inversions will result.
\item If $k\in U$ then make $k^0$ a new block of $\om$ so that exactly $\al_k$ new inversions result.
\een
By the way we have constructed this map, it will be weight preserving as long as it is well defined in that there exists a block, and position in a sequence if necessary, so that $k^0$ causes $\al_k$ inversions.  Since $\al\le\be(T)$, this will follow if the $k$th part of $\be(T)$ is the maximum number of inversions which can be caused by $k^0$ at the $k$th step.
This is an easy induction on $k$ using the definition of $\be(T)$ and the rules for when a new $m$-tuple of blocks is created by $k^0$.  So we leave it to the reader.
It is also routine to describe the inverse map by reversing the above algorithm step-by-step, so these details are omitted as well.
\eprf

\begin{exa}
To illustrate the procedure in the previous proof, suppose $n=m=3$, $T=\{2,3\}$, and $\al=(1,2,4)$.  We start with  $\om=(0)$.  Since $1\not\in T$, we must place $1^0,1^1,1^2$ in their own blocks.  And in order to obtain $\al_1=1$ inversion caused by $1^0$, we let 
$$
\om = (0 \mid 1^2/1^0/1^1).
$$
Now $2\in T$ so we will place the colors of $2$ in blocks with other elements.  If we wish to have $2^0$ involved with $\al_2=2$ inversions then this forces
$$
\om = (0 \mid 1^2 2^0/1^0 2^1/1^1 2^2).
$$
Finally we have $3\in T$ and $\al_3=4$ which results in 
$$
\om = (0\ 4^2 4^0 4^1 \mid 1^2 2^0/1^0 2^1/1^1 2^2).
$$
\end{exa}

\subsection{The Chan-Rhoades generalized coinvariant algebra}

We will use the abbreviation
$$
S[n,k]  = S[1,n,k]
$$
for the $q$-Stirling numbers in type $A$ and similarly for the ordered variant.  For any polynomial in $q$ we use the notation
$$
f(q)\vert_{q^m} = f(q^m).
$$

Haglund--Rhoades--Shimozono \cite{HRS:osp} introduced \textit{generalized coinvariant algebras}
  \[ R_{n, k} = \bbC[\bx]/\langle x_1^k, \ldots, x_n^k, e_n(n), e_{n-1}(n), \ldots, e_{n-k+1}(n)\rangle. \]
The classical coinvariant algebra of $A_{n-1}$ is $R_{n, n} = \RG_{1,n}$. 
Chan--Rhoades \cite{CR:Rnk} introduced generalized coinvariant algebras for $G(m, 1, n)$ with $m \geq 2$,
  \[ R_{n, k}^{(m)} = \bbC[\bx]/\langle x_1^{km+1}, \ldots, x_n^{km+1}, e_n(\bx^m), e_{n-1}(\bx^m), \ldots, e_{n-k+1}(\bx^m)\rangle \]
where $\bx^m$ refers to $x_1^m, \ldots, x_n^m$. Set $R_{n, k}^{(1)} = R_{n, k}$. Haglund--Rhoades--Shimozono \cite{HRS:osp} proved
\begin{equation}
\label{CR:A}
\rev_q \Hilb(R_{n, k}^{(1)}; q) = S^o[n, k] = [k]! h_{n-k}([1], [2], \ldots, [k]), 
\end{equation}
where if $f(q)$ is a polynomial of degree $d$ then
$\rev f(q) = q^d f(1/q)$ is the polynomial with reversed coefficients.

\begin{prop}
\label{CR:2}
  For $m \geq 2$, we have
    \[ \rev_q \Hilb(R_{n, k}^{(m)}; q) = S_{\CR}^o[m,n,k]
    =[km]!_m\ h_{n-k}([1],[m+1],\ldots,[km+1]). \]
\end{prop}  
  \begin{proof}
    After accounting for $q$-reversals, \cite[Cor.~4.11]{CR:Rnk} gives
    \begin{align*}
      \rev_q \Hilb(R_{n, k}^{(m)}; q)
      &= \sum_{i=0}^{n-k} \binom{n}{i} q^{n-k-i}\ [m]_q^{n-i}\ ([k]!\ S[n-i, k])\vert_{q^m} \\
      &= [km]!_m\ \sum_{i=0}^{n-k} \binom{n}{i}\ (q[m])^{n-k-i}\ S[n-i, k]\vert_{q^m}
    \end{align*}
    for $m \geq 2$. One may check directly that the final sum satisfies the same initial conditions and recurrence as $h_{n-k}([1], [m+1], \ldots, [km+1])$.  So the theorem follows from~\eqref{CR:def} and Theorem~\ref{S[m,n,k]:h}.
  \end{proof}

\begin{rem}
We can unify the  cases $m=1$ in equation~\eqref{CR:A} and $m\ge2$ in Proposition~\ref{CR:2} by noting that in both we have 
  \[ \rev_q \Hilb(R_{n, k}^{(m)}; q) = \left(\prod_{i=1}^k [d_i]\right) h_{n-k}([e_1^*], \ldots, [e_{k+1}^*]) \]
where the product is over the \textit{degrees} of $G(m,1,n)$ and $e_i^*$ are its coexponents (see \cite[Table 1]{SW:hdf}), both listed in increasing order. Here we use the convention that, for $m=1$, $d_i = i$ and $e_i^* = i-1$, so in particular $e_1^* = 0$.
\end{rem}

\section{Identities and an open problem}

\subsection{Identities}

Let us define $q$-Stirling numbers of the first kind for $G(m,1,n)$ consistent with \Cref{OS} for $m \geq 2$ to be
$$
s[m,n,k]= (-1)^{n-k} e_{n-k}([1],[m+1],\ldots[(n-1)m+1]).
$$
The following results about $s[m,n,k]$ and $S[m,n,k]$ can be derived using standard techniques from their expressions in terms of elementary and complete homogeneous functions, from the associated recursions, or for part (a) from \Cref{h:fall}.  So the proofs have been suppressed.
\begin{thm}
We have the following for $m \geq 2$.
\begin{enumerate}
    \item[(a)] $\dil t^n = \sum_{k=0}^n S[m,n,k] (t-[1])(t-[m+1])\cdots(t-[km-k+1])$,
    \item[(b)] $\dil\sum_{k=0}^n s[m,n, k] t^k = (t-[1])(t-[m+1])\cdots(t-[nm-m+1])$.
    \item[(c)] $\dil\sum_{n\ge0} S(m,n,k)\frac{x^n}{n!}
    =\frac{e^x}{k!} \left(\frac{e^{mx}-1}{m}\right)^k$,
    \item[(d)] $\dil\sum_{n\ge0} S(m,n,k)t^k\frac{x^n}{n!}
    =\exp(1+t(e^{mx}-1)/m)$,
    \item[(e)] $\dil\sum_{n\ge0} s(m,n,k)\frac{x^n}{n!}
    =\frac{1}{k! m^k (1+mx)^{1/m}}\left(\log(1+mx)  \right)^k$,
\item[(f)] $\dil\sum_{n\ge0} s(m,n,k) t^k\frac{x^n}{n!}
    =(1+mx)^{(1+t)/m}$.\hfill\qed
\end{enumerate}
\end{thm}

We also get the following result immediately from~\cite[Theorem 2.6]{SS:qStB}.

\begin{thm}
Consider infinite matrices  
$$
s_m = [s[m,n,k]]_{n,k\ge0}
$$
and 
$$
S_m=[S[m,n,k]]_{n,k\ge0}.
$$
Then $s_mS_m = I$ for all $m\ge1$.\hqed 
\end{thm}

\subsection{Open problem}

In Section~\ref{osn} we proved the alternating sum formula for $S^o[m,n,k]$ in Theorem~\ref{S^o[m]:alt} using the symmetric polynomial identity Theorem~\ref{h:fall}.  This identity can also be used to prove the corresponding formulas for $\St^o[m,n,k]$ and for $S_{\CR}^o[m,n,k]$ in Theorems~\ref{St^o:alt} and~\ref{S_CR:alt}, respectively.  However, we chose to give more combinatorial proofs in the latter two cases using sign-reversing involutions.  We have been unable to use this technique to prove Theorem~\ref{S^o[m]:alt}.  It would be interesting to do so.



\nocite{*}
\bibliographystyle{alpha}

\begin{thebibliography}{KAAM24}

\bibitem[ACK{\etalchar{+}}24]{ACKMR:sch}
Robert Angarone, Patricia Commins, Trevor Karn, Satoshi Murai, and Brendon
  Rhoades.
\newblock Superspace coinvariants and hyperplane arrangements.
\newblock arxiv:2404.17919, 2024.

\bibitem[Ars24]{ars:cim}
Hasan Arslan.
\newblock A combinatorial interpretation of {M}ahonian numbers of type {B}.
\newblock arXiv:2404.05099, 2024.

\bibitem[BGK]{BGK:qrS}
Eli Bagno, David Garber, and Takao Komatsu.
\newblock A $q,r$-analogue for the {S}tirling numbers of the {S}econd kind of
  type $b$.
\newblock Preprint.

\bibitem[BS17]{BS:ai}
Carolina Benedetti and Bruce~E. Sagan.
\newblock Antipodes and involutions.
\newblock {\em J. Combin. Theory Ser. A}, 148:275--315, 2017.

\bibitem[CR20]{CR:Rnk}
Kin Tung~Jonathan Chan and Brendon Rhoades.
\newblock Generalized coinvariant algebras for wreath products.
\newblock {\em Adv. in Appl. Math.}, 120:102060, 61, 2020.

\bibitem[DZ24]{ZD:sii}
Ming-Jian Ding and Jiang Zeng.
\newblock Some identities involving {$q$}-{S}tirling numbers of the second kind
  in type {B}.
\newblock {\em Electron. J. Combin.}, 31(1):Paper No. 1.36, 24, 2024.

\bibitem[HRS18]{HRS:osp}
James Haglund, Brendon Rhoades, and Mark Shimozono.
\newblock Ordered set partitions, generalized coinvariant algebras, and the
  delta conjecture.
\newblock {\em Adv. Math.}, 329:851--915, 2018.

\bibitem[KAAM24]{KAAM:cqm}
Ali Kessouri, Moussa Ahmia, Hasan Arslan, and Salim Mesbahi.
\newblock Combinatorics of $q$-{M}ahonian numbers of type {$B$} and
  log-concavity, 2024.
\newblock arXiv:2408.02424.

\bibitem[OS80]{OS:urg}
Peter Orlik and Louis Solomon.
\newblock Unitary reflection groups and cohomology.
\newblock {\em Invent. Math.}, 59(1):77--94, 1980.

\bibitem[Sag20]{Sag:aoc}
Bruce~E. Sagan.
\newblock {\em Combinatorics: the art of counting}, volume 210 of {\em Graduate
  Studies in Mathematics}.
\newblock American Mathematical Society, Providence, RI, [2020] \copyright
  2020.

\bibitem[SS24]{SS:qStB}
Bruce~E. Sagan and Joshua~P. Swanson.
\newblock {$q$}-{S}tirling numbers in type {$B$}.
\newblock {\em European J. Combin.}, 118:Paper No. 103899, 35, 2024.

\bibitem[ST54]{ST:fur}
G.~C. Shephard and J.~A. Todd.
\newblock Finite unitary reflection groups.
\newblock {\em Canad. J. Math.}, 6:274--304, 1954.

\bibitem[SW23]{SW:hdf}
Joshua~P. Swanson and Nolan~R. Wallach.
\newblock Harmonic differential forms for pseudo-reflection groups {II}.
  {B}i-degree bounds.
\newblock {\em Comb. Theory}, 3(3):Paper No. 17, 43, 2023.

\bibitem[Zas81]{zas:grs}
Thomas Zaslavsky.
\newblock The geometry of root systems and signed graphs.
\newblock {\em Amer. Math. Monthly}, 88(2):88--105, 1981.

\end{thebibliography}

\end{document}